\documentclass[11pt]{article}
\usepackage[utf8]{inputenc}

\usepackage{color}    
\usepackage{epsfig} 
\usepackage{epsf}            
\usepackage{float}           

\usepackage{tikz}
\usetikzlibrary{matrix}
\usetikzlibrary{arrows.meta}
\usetikzlibrary{arrows}
\usetikzlibrary{math}

\usepackage[margin=1cm]{caption}

\usepackage{amsthm}
\usepackage{graphics}
\usepackage{amsfonts}
\usepackage{amsxtra}
\usepackage{amsmath}
\usepackage{a4}
\usepackage{amssymb}
\newtheorem{theorem}{Theorem}
\newtheorem{lemma}[theorem]{Lemma}
\newtheorem{corollary}[theorem]{Corollary}

\newtheorem{proposition}[theorem]{Proposition}

\theoremstyle{definition}
\newtheorem{example}[theorem]{Example}
\newtheorem{remark}[theorem]{Remark}

\newtheorem{question}[theorem]{Question}

\numberwithin{theorem}{section}
\numberwithin{equation}{section}
\numberwithin{figure}{section}

\renewcommand{\bullet}{\circ}

\newcommand{\mex}{\operatorname{mex}}
\newcommand{\rank}{\operatorname{rank}}
\newcommand{\OO}{\mathcal{O}}
\newcommand{\height}{\operatorname{height}}

\newcommand{\Ch}[2]{\ensuremath{\begin{pmatrix} #1 \\ #2 \end{pmatrix}}}

\newcommand{\union}{\cup}

\newcommand{\BGW}{Galton-Watson}

\newcommand{\PP}{\mathbb{P}}

\newcommand{\NN}{\mathbb{N}}

\newcommand{\R}{\mathbb{R}}

\newcommand{\cC}{\mathcal{C}}
\newcommand{\cA}{\mathcal{A}}

\newcommand{\cB}{\mathcal{B}}

\newcommand{\cP}{\mathcal{P}}

\newcommand{\cG}{\mathcal{G}}
\newcommand{\cD}{\mathcal{D}}

\newcommand{\cN}{\mathcal{N}}

\newcommand{\bp}{\mathbf{p}}

\title{Extended Sprague-Grundy theory for 
locally finite games, and applications to 
random game-trees
}
\author{
\textbf{James B.~Martin}
\\
\textit{University of Oxford}
\\
\texttt{martin@stats.ox.ac.uk}
}
\date{}

\begin{document}

\maketitle

\begin{center}
\textit{For a collection of papers in memory of 
Elwyn~Berlekamp (1940-2019), John~Conway (1937-2020), and Richard~Guy (1916-2020)}
\end{center}

\begin{abstract}
The Sprague-Grundy theory 
for finite games without cycles 
was extended to general finite games 
by Cedric Smith and by Aviezri Fraenkel and coauthors. 
We observe that the same framework
used to classify finite games
also covers the case of locally finite games
(that is, games where any position has only finitely
many options). 
In particular, any locally finite game is equivalent to
some finite game.
We then study cases where the directed
graph of a game is chosen randomly, and is 
given by the tree of a \BGW{}
branching process. Natural families of offspring 
distributions display a surprisingly wide range of behaviour.
The setting shows a nice interplay
between ideas from combinatorial game theory
and ideas from probability. 
\end{abstract}

\section{Introduction}

Among the plethora of beautiful and intriguing examples
to be found in
Elwyn Berlekamp, John Conway, and Richard Guy's 
 \textit{Winning Ways}
is the game of 
\textsc{Fair Shares and Varied Pairs}
(\cite[Chapter 12]{WinningWays2}).
The game is played with some number of almonds, which are arranged into heaps. 
A move of the game consists of either
\begin{itemize}
    \item dividing any heap into two or more equal-sized heaps (hence ``fair shares"); or
    \item uniting any two heaps of different sizes (hence ``varied pairs").
\end{itemize}
The only position from which no move is possible is the one where all the almonds
are completely separated into heaps of size $1$. When that position is reached, the
player who has just moved is the winner. 

\textsc{Fair Shares and Varied Pairs} is a \textbf{loopy} game: the directed graph
of game positions has cycles, so the game can return to a previously visited position. 
The way in which the loopiness manifests itself depends on the number of almonds:
\begin{itemize}
    \item With 3 or fewer almonds, there are no cycles. The game is \textbf{non-loopy}.
    \item With 4 to 9 almonds, the graph has loops, but all positions are equivalent to finite nim heaps. Hence 
    in any position, either the first player has a winning
    strategy, or the second player has a winning strategy;
    furthermore, the same is true for the (disjunctive) sum of any two positions, or for the sum of a position with a nim heap. Berlekamp, Conway and Guy call such behaviour
    \textbf{latently loopy}. ``This kind of loopiness is really illusory; unless the winner
    wants to take you on a trip, you won't notice it."
    \item With 10 almonds, still any position has either a forced win for the first player or a forced win for the second player. However, now there exist some \textbf{patently loopy} positions which are not equivalent to finite nim heaps. If one takes the sum of two such positions, or the sum of such a position with a nim heap, one can obtain a game where neither player has a winning strategy -- 
    the game is drawn with best play. 
    \item With 11 or more almonds, there exist \textbf{blatantly loopy} positions where the 
    game is drawn with best play.
\end{itemize}

In this article we explore similar themes, 
but we concentrate particularly on 
cases where the possibility of draws
comes not necessarily from cycles in the game-graph, but 
instead from infinite paths. (Although the game-graph
may be infinite, from any given position there will be
be only finitely many possible moves.)

We also focus on situations where the directed graph
of the game is chosen at random.
The randomness is only in the choice of the graph, i.e.\ of the ``rules of the game". All the games themselves will be combinatorial games in the usual sense, with full information and with no randomness. 

Here is an example. We will consider a population 
where each individual reproduces with 
some given probability $p\in(0,1)$. 
If an individual reproduces, it has $4$ children. 
We start with a single individual (the ``root"). 
With probability $1-p$, the root has no children, and with 
probability $p$, the root has $4$ children, forming generation $1$. If the root does have children, 
then in turn each of those children itself has no children with probability $1-p$, 
and has $4$ children with probability $p$. The collection of those families forms
generation $2$, whose members again go on to reproduce in the same way, and so on. 
All the decisions are made independently. 
From the family tree of this process, we form a directed
graph by taking the individuals as vertices, 
and adding an arc from each vertex to each of its
children. This is an example of a \textit{\BGW{} tree}
(or \textit{Bienaym{\'e} tree}). In the game played on this tree, from every position
there are either $0$ or $4$ possible moves. Again we consider
normal play -- if there are no moves possible from 
a position, the next player to move loses. 

Note for example that the tree could be trivial: with probability $1-p$ it consists
of just a single vertex. Or it could be larger but finite (its size can be any integer which is congruent to $1 \bmod 4$), 
as shown in the example in Figure \ref{fig:gametree}.
But the tree can also be infinite. 

\tikzmath{\hsep=1; \vsep1=1.2; \vsep2=0.8; \vsep3=0.4;}

\begin{figure}[ht]
\begin{center}
\begin{tikzpicture}
\node (A) at (0, 0) {$\bullet$};

\node (B1) at (-1.5*\vsep1, -\hsep) {$\bullet$};
\node (B2) at (-0.5*\vsep1, -\hsep) {$\bullet$};
\node (B3) at (0.5*\vsep1, -\hsep) {$\bullet$};
\node (B4) at (1.5*\vsep1, -\hsep) {$\bullet$};

\node (C1) at (-1.5*\vsep1-1.5*\vsep2, -2*\hsep) {$\bullet$
};
\node (C2) at (-1.5*\vsep1-0.5*\vsep2, -2*\hsep) {$\bullet$};
\node (C3) at (-1.5*\vsep1+0.5*\vsep2, -2*\hsep) {$\bullet$};
\node (C4) at (-1.5*\vsep1+1.5*\vsep2, -2*\hsep) {$\bullet$};
\node (C5) at (1.5*\vsep1-1.5*\vsep2, -2*\hsep) {$\bullet$};
\node (C6) at (1.5*\vsep1-0.5*\vsep2, -2*\hsep) {$\bullet$};
\node (C7) at (1.5*\vsep1+0.5*\vsep2, -2*\hsep) {$\bullet$};
\node (C8) at (1.5*\vsep1+1.5*\vsep2, -2*\hsep) {$\bullet$};

\node (D1) at (1.5*\vsep1-0.5*\vsep2-1.5*\vsep3, -3*\hsep) {$\bullet$};
\node (D2) at (1.5*\vsep1-0.5*\vsep2-0.5*\vsep3, -3*\hsep) {$\bullet$};
\node (D3) at (1.5*\vsep1-0.5*\vsep2+0.5*\vsep3, -3*\hsep) {$\bullet$};
\node (D4) at (1.5*\vsep1-0.5*\vsep2+1.5*\vsep3, -3*\hsep) {$\bullet$};
\node (D5) at (1.5*\vsep1+1.5*\vsep2-1.5*\vsep3, -3*\hsep) {$\bullet$};
\node (D6) at (1.5*\vsep1+1.5*\vsep2-0.5*\vsep3, -3*\hsep) {$\bullet$};
\node (D7) at (1.5*\vsep1+1.5*\vsep2+0.5*\vsep3, -3*\hsep) {$\bullet$};
\node (D8) at (1.5*\vsep1+1.5*\vsep2+1.5*\vsep3, -3*\hsep) {$\bullet$};

\draw[-Latex, shorten >=-4pt] 
(A) edge (B1) 
(A) edge (B2)
(A) edge (B3)
(A) edge (B4)
(B1) edge (C1)
(B1) edge (C2)
(B1) edge (C3)
(B1) edge (C4)
(B4) edge (C5)
(B4) edge (C6)
(B4) edge (C7)
(B4) edge (C8)
(C6) edge (D1)
(C6) edge (D2)
(C6) edge (D3)
(C6) edge (D4)
(C8) edge (D5)
(C8) edge (D6)
(C8) edge (D7)
(C8) edge (D8)
;


\end{tikzpicture}
\end{center}
\caption{
\label{fig:gametree}
An example of a finite directed graph that
could arise from the \BGW{} tree model
considered in the introduction with out-degrees $0$
and $4$.}
\end{figure}
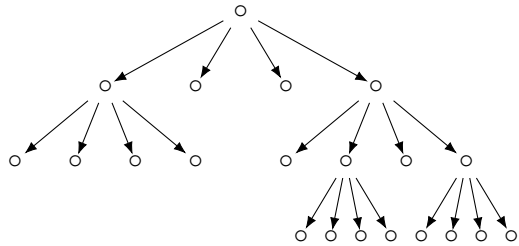

The game played with such a tree as its game-graph displays very interesting parallels with that of \textsc{Fair Shares and Varied Pairs} described above. The behaviour depends on the value of the parameter $p$. We will find that there are thresholds 
$a_0=1/4$, $a_1\approx 0.52198$, $a_2=5^{3/4}/4\approx 0.83593$
such that:
\begin{itemize}
    \item For $p\leq a_0$, the tree is finite with probability $1$.
    \item For $a_0<p\leq a_1$, there is positive probability that the tree is infinite. However, 
    with probability $1$, all its positions are equivalent to finite nim heaps, and so 
    in particular every position has a winning strategy
    for one or the other of the players.
    \item For $a_1<p\leq a_2$, still with probability $1$ every position has a winning strategy for one player or the other.
    However, there is now positive probability that the tree has positions which are not equivalent to finite nim heaps. The sum of two such games, or the sum of such a game with a nim heap, may be drawn with best play.
    \item For $p>a_2$, with positive probability the tree has positions which are drawn with best play.
\end{itemize}

\subsection{Background and outline of results}
The equivalence of any finite loop-free impartial game
to a nim heap was shown independently
by Roland Sprague and by Patrick Grundy in the 1930s.
Richard Guy was a key figure in developing and
broadening the scope of the Sprague-Grundy theory
in the next couple of decades, notably for example 
in his 1956 paper with Cedric Smith \cite{GuySmith}.

An extension of the Sprague-Grundy theory 
to finite games which may contain cycles
was first described by Smith 
\cite{SmithGraphs} and was developed extensively in 
a series of works by Aviezri Fraenkel and coauthors
(for example \cite{FraenkelTassa, FraenkelPerl, FraenkelYesha}). As well as \textit{finite-rank} 
games that are equivalent to nim heaps, one now
additionally has \textit{infinite-rank} games
which are not equivalent to nim heaps. 
The ``extended Sprague-Grundy value"
(or ``loopy nim value") of such a game is written
in the form $\infty(\cA)$, where $\cA\subset\NN$
is the set of nim values of the game's 
finite-rank options. These infinite-rank games
may either be first-player wins (if $0\in\cA$)
or draws (if $0\notin\cA)$.
Again, we have equivalence between two games
if and only if they have the same (extended) Sprague-Grundy
value. 

In \cite{SmithGraphs} Smith already envisages 
extensions of the theory to infinite games,
involving ordinal-valued Sprague-Grundy functions. 
An extension of a different sort to infinite graphs was done by Fraenkel and Rahat \cite{FraenkelRahatpathbounded}, who extend the finite non-loopy
Sprague-Grundy theory to infinite games 
which are \textit{locally path-bounded}, in the sense
that for any vertex of the game-graph, the set of paths starting at that vertex has finite maximum length. 

In this paper we observe that the extended
Sprague-Grundy values which classify finite games
are also enough to classify the class of 
\textit{locally finite games}, in which every position 
has finitely many options. As a result, any such 
locally finite (perhaps cyclic) game is equivalent to a finite (perhaps cyclic) game. 

We then focus in particular on applying the theory
to games whose directed graph is given by a
\textit{\BGW{} tree}, of which the
$0$-or-$4$ tree described in the previous section
is an example. \BGW{} trees provide an extremely
natural model of a random game-tree. They have a self-similarity which can be described as follows:
the root individual has a random number of children
(distributed according to the \textit{offspring distribution}), and then conditional on that number of children, the sub-trees of descendants of each of those children
are independent and have the same distribution as the 
original tree. 

Games on \BGW{} trees (including normal play, 
mis\`{e}re play, and other variants) are
studied by Alexander Holroyd and the current author
in \cite{GaltonWatsongames}. 
There, a particular focus was on 
determining which offspring distributions 
give positive probability of a draw,
and on describing
the type of phase transition that occurs between
the sets of distributions with and without draws. 
In this paper we concentrate on normal play; 
but, armed with the extended Sprague-Grundy theory,
we can investigate, for example, whether
infinite-rank positions occur in games without draws
(the case analogous to Berlekamp, Conway and
Guy's ``patently loopy" behaviour described above). 
This setting shows a very nice interplay between
ideas from combinatorial game theory and
ideas from probability.

One tool on which we rely heavily is the 
study of the behaviour of the game-graph when
the set $\cP$ of its second-player-winning positions is removed. This reduction behaves especially
nicely in the \BGW{} setting.
For example, if we take 
a \BGW{} tree for which draws have probability $0$,
condition the root to be a first-player win, and remove
the set $\cP$, then the remaining component connected
to the root is again a \BGW{} tree, with a new
offspring distribution. Combining iterations of this
procedure with recursions involving the probability generating
function of the offspring distribution yields
a lot of information about the infinite-rank positions
that can occur in the tree. 

We finish by presenting three particular examples
of families of offspring distribution: 
the Poisson case, the geometric case, 
and the $0$-or-$4$ case described above. 
In these examples alone we see a surprisingly wide
variety of different types of behaviour. 

We now briefly describe the organisation of the paper. 

In Section \ref{sec:basicESG} we describe 
the extended Sprague-Grundy theory for
locally finite games. Although the setting is new, 
the results can be written in a form
which is almost identical to that of the finite case.
We proceed in a way that closely parallels the presentation
of Siegel from Section IV.4 of \cite{Siegelbook}
(with some variations of notation). 
The proofs given in \cite{Siegelbook} 
also carry over to the current setting essentially unchanged, and for that reason we do not reproduce
them here. A reader who is not already familiar with the extended Sprague-Grundy theory for finite games may 
like to start with that section of 
\cite{Siegelbook} before reading on further here. 

In Section \ref{sec:reductions}
we discuss the operation of removing $\cP$-positions
from a locally finite game, and examine
its effect on the Sprague-Grundy values of 
the positions which remain. For the particular
case of trees, we give an interpretation involving
\textit{mex labellings} 
(labellings of the vertices of the tree by natural numbers
which obey mex recursions at each vertex). 

In Section \ref{sec:GW}, we introduce 
games on \BGW{} trees, and develop
the analysis via graph reductions and generating
function recursions. 

Finally, examples of particular offspring 
distributions are studied in Section~\ref{sec:examples}.

\section{Extended Sprague-Grundy theory for games with infinite paths}\label{sec:basicESG}
In this section we introduce
basic notation and definitions, 
and then describe the extended Sprague-Grundy
theory for locally finite games. 
The results look identical to those that have 
previously been written for the case of finite games. 
Proofs of these results, written for the case of finite
games but equally applicable here, 
can be found in Section IV.4 of \cite{Siegelbook}.
However, note that formally speaking, the content
of the results is different; this is not just because the scope
of the statements is broader, but also because the definition of equivalence
is different (see the discussion in Section \ref{sec:equivalence}). 
\subsection{Directed graphs and games}
We will represent impartial games by directed graphs. If $V$ is a directed graph, we call the
vertices of $V$ \textbf{positions}. If there is an arc from $x$ to $y$ in $V$,
we write $y\in \Gamma(x)$ (or $y\in \Gamma_V(x)$ if we want to specify the graph $V$) -- here $\Gamma(x)$ is the set of \textbf{options} (i.e.\ out-neighbours) of $x$.
We say that the graph $V$ is \textbf{locally finite} if all its
vertices have finite out-degree; that is, $\Gamma(x)$ is a finite set for each vertex $x$. We may be deliberately loose in using the same symbol $V$ to refer both to the graph and to the set of vertices of the graph.

Informally, we consider two-player games with alternating turns; each turn consists of moving from a position $x$ to a position $y$, where $y\in \Gamma(x)$.
We consider normal play: if we reach a \textbf{terminal} position, 
meaning a vertex with outdegree $0$, then the next player to move loses. 
Since the graphs we consider may have cycles or infinite paths, it may be
that play continues for ever without either player winning. 

Formally, a \textbf{locally finite game} is a pair $G=(V,x)$ where $V$ is a 
locally finite directed graph (which is allowed to contain
cycles) and $x$ is a vertex of $V$. 
We will often write just $x$ instead of $(V,x)$ when the graph $V$ is understood. 
For example, for the outcome function $\OO$, the Sprague-Grundy function $\cG$, and the rank function
(all defined below), we will often write $\OO(x)$, $\cG(x)$,
and $\rank(x)$, rather than $\OO((V,x))$, $\cG((V,x))$,
and $\rank((V,x))$. We use the fuller notation when we need to consider more than one 
graph simultaneously (for example when considering disjunctive sums of games,
or when considering operations which reduce a graph by removing some of its vertices). 

Let $V$ be a directed graph and $o$ a vertex of $V$.
If $o$ has in-degree $0$, and if for every $x\in V$,
there exists a unique directed walk from $o$ to $x$, 
then we say that $V$ is a \textbf{tree} with 
\textbf{root} $o$. If $x$ and $y$ are vertices
of a tree $V$ with $y\in \Gamma_V(x)$, we may say that
$y$ is a \textbf{child} of $x$ in $V$.
We write $\height(x)$ for the \textbf{height}
of $x$, which is the number of arcs
in the path from $o$ to $x$. 

\subsection{Outcome classes}
For a graph $V$, each position $x\in V$ falls into one of three outcome classes:
\begin{itemize}
    \item 
If the first player has a winning strategy from $x$,
then we write $x\in\cN$, or $\OO(x)=\cN$, and say that $x$ is an $\cN$-position.
\item
If the second player has a winning strategy from $x$,
then we write $x\in\cP$, or $\OO(x)=\cP$, and say that $x$ is an $\cP$-position.
\item
If neither player has a winning strategy from $x$, so that
with optimal play the game continues for ever without reaching a terminal position, 
we write 
$x\in\cD$, or $\OO(x)=\cD$, and say that $x$ is an $\cD$-position.
\end{itemize}

\begin{theorem}
Let $V$ be a locally finite graph, and $x\in V$.
\begin{itemize}
    \item 
    $x$ is a $\cP$-position iff every $y\in \Gamma(x)$ is an $\cN$-position.
    \item
    $x$ is an $\cN$-position iff some $y\in \Gamma(x)$ is a $\cP$-position.
    \item   
    $x$ is a $\cD$-position iff no $y\in \Gamma(x)$ is a $\cP$-position, but some $y\in \Gamma(x)$
    is a $\cD$-position.
\end{itemize}
\end{theorem}

\subsection{Disjunctive sums and equivalence between games}
\label{sec:equivalence}
Let $V$ and $W$ be directed graphs. We define
$V\times W$ to be the directed graph whose vertices are
$\{(x,y), x\in V, y\in W\}$, and which has an arc from 
$(u,v)$ to $(x,y)$ iff either $u=x$ and $y\in \Gamma_W(v)$, 
or $x\in \Gamma_V(u)$ and $v=y$. If $V$ and $W$ are both locally finite,
then so is $V\times W$

If $G=(V,x)$ and $H=(W,y)$ are locally finite games, 
we define their (disjunctive) \textbf{sum} $G+H$
to be the locally finite game $(V\times W, (x,y))$.

We have the following interpretation. A position
of $V\times W$ is an ordered pair of a position of $V$ and a position of $W$.
To make a move in the sum of games, from position $(x,y)$ of $V\times W$,
one must either move from $x$ to one of its options in $V$,
or from $y$ to one of its options in $W$ (and not both). 
The position $(x,y)$ is terminal for $V\times W$ iff $x$ is terminal for
$V$ and $y$ is terminal for $W$.

Now we define equivalence between two locally finite games $G$ and $H$. 
The games $G$ and $H$ are said to be \textbf{equivalent},
denoted by $G=H$,
if $\OO(G+X)=\OO(H+X)$ for every locally finite game $X$.

Note here that we have defined equivalence within the class of locally finite
games: we required the equality to hold for every locally finite game $X$.
The definition (and the meaning of the results below) 
would be different if $X$ ranged over a different set. 
However, it will follow from the extended Sprague-Grundy theory below 
that this equivalence extends both the equivalence within the class
of finite loopfree graphs, and that within the class of finite graphs. 
That is, two finite games are equivalent within the class of
finite games iff they are equivalent with the class of locally finite games; 
also two finite loopfree games are equivalent within the class
of finite loopfree games iff they are equivalent within the class of finite games. 



\subsection{The rank function and the Sprague-Grundy function}
Let $V$ be a locally finite directed graph. We recursively define $\cG_n(x)$ for $x\in V$ and 
$n\geq 0$ as follows.
First, let 
\[
\cG_0(x)=\begin{cases}
0,&\text{if $x$ is terminal;}\\
\infty,&\text{otherwise.}
\end{cases}
\]
Then for $n\geq 1$ and given $x$, write $m=\mex\{\cG_{n-1}(y), y\in \Gamma(x)\}$, and
let
\[
\cG_{n}(x)=\begin{cases}
m,&\parbox[t][][t]{0.7\linewidth}{
if for each $y\in\Gamma(x)$,
either $\cG_{n-1}(y)\leq m$, 
or there is $z\in\Gamma(y)$ with $\cG_{n-1}(z)=m$;
}\\
\infty,&\text{otherwise.}
\end{cases}
\]

\begin{proposition}
\label{prop:SGsetup}
Let $x\in V$. Then either:
\begin{itemize}
    \item 
    $\cG_n(x)=\infty$ for all $n$; or
    \item
    there exist $m$ and $n_0$ such that
    \[
    \cG_n(x)=\begin{cases}
    \infty,&\text{if } n< n_0;\\
    m,&\text{if }n\geq n_0.
    \end{cases}
    \]
\end{itemize}
\end{proposition}

In the light of Proposition \ref{prop:SGsetup}, we can now define
the \textbf{extended Sprague-Grundy function} $\cG$ in the case of a 
locally finite graph $V$. Let $x\in V$. If the second case of 
Proposition \ref{prop:SGsetup} holds, and $\cG_n(x)=m$ for all sufficiently large $n$, then $\cG(x)=m$. Otherwise, we write
\[
\cG_n(x)=\infty(\cA),
\]
where $\cA$ is the finite set defined by
\[
\cA=\{a\in\NN: \cG(y)=a \text{ for some }y\in \Gamma(x)\}.
\]
We then define the \textbf{rank} of $x$,
written $\rank(x)$,
to be the least 
$n$ such that $\cG_n(x)$ is finite, or $\infty$ if no such $n$ exists.
(Hence the finite-rank vertices are
those $x$ with $\cG(x)=m\in\NN$, 
while the infinite-rank vertices are those $x$
with $\cG(x)=\infty(\cA)$ for some $\cA\subset\NN$.)

Some examples of extended Sprague-Grundy values can be found
in Figure \ref{fig:reduction-counterexample}.

The extended Sprague-Grundy value of $x$ determines its outcome $\OO(x)$:
\begin{theorem}\label{thm:SpragueGrundyoutcomes}
$\,$
\begin{itemize}
    \item[(a)] $\cG(x)=0$ iff $\OO(x)=\cP$.
    \item[(b)] If $\cG(x)$ is a positive integer, then $\OO(x)=\cN$.
    \item[(c)] If $\cG(x)=\infty(\cA)$ for a set $\cA$ with $0\in\cA$, then 
    $\OO(x)=\cN$.
    \item[(d)] $\cG(x)=\infty(\cA)$ for some $\cA$ with $0\notin\cA$
    iff $\OO(x)=\cD$.
\end{itemize}
\end{theorem}

Theorem \ref{thm:SpragueGrundyoutcomes} tells us that the Sprague-Grundy value 
of a position determines its outcome class. In fact, much more is true: 
the Sprague-Grundy values of two games determines the Sprague-Grundy value, 
and hence the outcome class, of their sum. The algebra of
the Sprague-Grundy values is the same as in the case of finite loopy graphs,
and full details can be found at the end of Section IV.4 of \cite{Siegelbook}. Again
the proofs carry over unchanged to the locally finite setting. 
We note a few particular consequences:
\begin{theorem}\label{thm:algebra}
Let $G$ and $H$ be locally finite games. 
\begin{itemize}
    \item[(a)] $G+H$ has infinite rank iff at least one of $G$ and $H$ have infinite rank.
    \item[(b)] If both $G$ and $H$ have infinite rank than $\cG(G+H)=\infty(\emptyset)$,
    and in particular $\OO(G+H)=\cD$.
    \item[(c)] If $\cG(G)=m\in\NN$, then $G$ is equivalent to $*m$, a nim heap of size $m$. 
    \item[(d)] $G$ and $H$ are equivalent iff $\cG(G)=\cG(H)$.
\end{itemize}
\end{theorem}

\begin{corollary}\label{corr:equivalence}
Every locally finite game
is equivalent to some finite game.
\end{corollary}

We finish the section by recording the following 
consequence of the 
construction
of the extended Sprague-Grundy function,
in a form which will be useful for later reference.

\begin{proposition}\label{prop:usefulSGfacts}
Let $V$ be a locally finite graph, and $x\in V$.
Then the following are equivalent:
\begin{itemize}
    \item[(a)]
    $\rank(x)\leq n$ and $\cG(x)=m$.
    \item[(b)]
    The following two properties hold:
    \begin{itemize}
        \item[(i)]
        For each $i$ with $0\leq i\leq m-1$,
        there exists $y_i\in \Gamma(x)$
        such that $\rank(y)<n$ and $\cG(y_i)=i$.
        \item[(ii)]
        For all $y\in\Gamma(x)$,
        either $\rank(y)<n$ and $\cG(y)< m$, 
        or there is $z\in \Gamma(y)$
        with $\rank(z)<n$ and $\cG(z)=m$.
    \end{itemize}       
\end{itemize}
\end{proposition}

\section{Reduced graphs}\label{sec:reductions}

Let $k\geq 0$. We will say that a locally finite directed graph $V$ is \textbf{$k$-stable}
if whenever $x\in V$ has infinite rank, i.e.\
whenever $\cG((V,x))=\infty(\cA)$ for some $\cA$, 
then $\{0,1,\dots, k\}\subseteq \cA$.

Note that by Theorem \ref{thm:SpragueGrundyoutcomes}(d), being $0$-stable is equivalent to being \textbf{draw-free}: every position of $V$ has a winning strategy either for the first player or for the second player. 

Let $\cP_V$ be the set of $\cP$-positions of the graph $V$,
i.e.\ those $x\in V$ with $\cG((V,x))=0$. 
Consider the graph $R(V):=V\setminus\cP_V$ which results from removing the $\cP$-positions from $V$ (and retaining all arcs between remaining vertices). 
More generally, for $k\geq 1$ let $R^{k}(V)$ be the graph resulting from removing all 
vertices $x$ with $\cG((V,x))<k$.

\begin{theorem}\label{thm:reduction}
Let $V$ be a locally finite directed graph,
and let $x\in R(V)$.
\begin{itemize}
\item[(a)]
If $x$ has finite rank in $V$, then also
$x$ has finite rank in $R(V)$;
specifically,
\[
\cG((R(V), x))=\cG((V,x))-1.
\]
\item[(b)]
Suppose additionally that $V$ is draw-free.  
If $x$ has infinite rank in $V$, then
also $x$ has infinite rank in $R(V)$;
specifically, if $\cG((V,x))=\infty(\cA)$ for some $\cA$ (in which case necessarily $0\in\cA$), 
then
\[
\cG((R(V), x))=\infty(\cA-1),
\]
where $\cA-1$ denotes the set $\{a\geq 0: a+1\in\cA\}$.
\end{itemize}
\end{theorem}

If $V$ is not draw-free, then the conclusion 
of part (b) may fail; removing the $\cP$-positions
may convert infinite-rank vertices to 
finite-rank vertices (either $\cP$-positions
or finite-rank $\cN$-positions). See 
Figure \ref{fig:reduction-counterexample}
for an example.

\begin{figure}[ht]
\begin{center}
\begin{tikzpicture}[>=Latex]
\node[draw, circle, label=west:
\textcolor{red}{$\infty(\{1\})$}
] (A) at (0, 0) {a};
\node[draw, circle, label=west:
\textcolor{red}{$\infty(\{2\})$}
] (B) at (0, 2) {b};
\node[draw, circle,
label=east:
\textcolor{red}{$2$}] (C) at (2, 2) {c};
\node[draw, circle,
label=east:
\textcolor{red}{$1$}] (D) at (2, 0) {d}; 
\node[draw, circle,
label=east:
\textcolor{red}{$0$}] (E) at (4, 1) {e}; 

\path[->, thick] 
(A) edge (B) 
(A) edge (D)
(B) edge (C)
(C) edge (D)
(D) edge (E)
(C) edge (E)
(B) edge [out=135, in=45, loop] (B)
;

\node[draw, circle, label=west:
\textcolor{red}{$1$}
] (A1) at (7, 0) {a};
\node[draw, circle, label=west:
\textcolor{red}{$\infty(\{1\})$}
] (B1) at (7, 2) {b};
\node[draw, circle,
label=east:
\textcolor{red}{$1$}] (C1) at (9, 2) {c};
\node[draw, circle,
label=east:
\textcolor{red}{$0$}] (D1) at (9, 0) {d}; 

\path[->, thick] 
(A1) edge (B1) 
(A1) edge (D1)
(B1) edge (C1)
(C1) edge (D1)
(B1) edge [out=135, in=45, loop] (B1)
;

\end{tikzpicture}
\end{center}
\caption{
\label{fig:reduction-counterexample}
The conclusion of Theorem \ref{thm:reduction}(b)
may fail when the graph is not draw-free. 
Here, removing the unique $\cP$-position $e$
from the graph on the left, to give the graph on the right,
converts the position $a$ from infinite rank to 
finite rank. The extended Sprague-Grundy values
are shown by the nodes in red.
}
\end{figure}
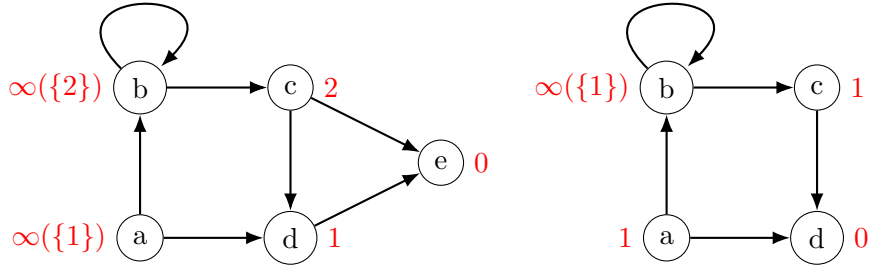

\begin{corollary}\label{corr:Vk} 
Let $k\geq 1$.
\begin{itemize}
\item[(a)]
Suppose that
$V, R(V), \dots, R^{(k)}(V)$ are all draw-free.
Then $R^{(k+1)}(V)=R(R^{(k)}(V))$.
\item[(b)]
$V$ is $k$-stable iff
$V, R(V), \dots, R^{(k)}(V)$ are all draw-free.
\end{itemize}
\end{corollary}

\begin{proof}[Proof of Theorem \ref{thm:reduction}]
(a) For the first part, we use induction
on the rank of $x$ in $V$.
We claim that if $x\in R(V)$ 
has $\rank((V,x))=n$ 
and $\cG((V,x))=m>0$, then $\rank((R(V), x)\leq n$
and $\cG((R(V),x))=m-1$.

Any $x$ with rank $0$ in $V$ is in $\cP_V$ 
and hence is not a vertex of $R(V)$,
so the claim holds vacuously for $x$ with $\rank((V,x))=0$.

Now for $n>0$, suppose the claim holds for all $x$ with $\rank((V,x))<n$,
and consider $x\in R(V)$ with $\rank((V,x))=n$
and $\cG((V,x))=m$.

From Proposition \ref{prop:usefulSGfacts}
we have the following properties:
\begin{itemize}
    \item[(i)]
    For each $i=0, \dots, m-1$, there exists $y_i
    \in \Gamma_V(x)$ such that $\rank((V,y_i))<n$
and    
    $\cG((V,y_i))=i$;
    \item[(ii)]
    For all $y\in \Gamma_V(x)$, either 
    $\rank((V,y))<n$ and $\cG((V,y))< m$, or
    there is $z\in \Gamma_V(y)$ with $\rank((V,z))<n$ and $\cG((V,z))=m$.
\end{itemize}
Applying the induction hypothesis we get:
\begin{itemize}
    \item[(i)] 
    For each $i=1, \dots, m-1$, there exists $y_i
    \in \Gamma_{R(V)}(x)$ such that $\rank((R(V),y_i))<n$
    and 
    $\cG((R(V), y_i))=i-1$;
    \item[(ii)]
    For all $y\in \Gamma_{R(V)}(x)$,
    either $\rank(R(V),y)<n$ and $\cG((R(V),y))<m-1$,
    or there is $z\in \Gamma_{R(V)}(y)$
    with $\rank((R(V),z))<n$ and $\cG((R(V), z))=m-1$.
\end{itemize}
Using Proposition \ref{prop:usefulSGfacts} again
we conclude that 
$\rank((R(V), x))\leq n$ and
$\cG((R(V), x))=m-1$,
completing the induction step.

(b) Now we suppose that in addition $V$ is draw-free. 
We first want to show that if $x$ has
finite rank in $R(V)$, then it also has
finite rank in $V$.
In this case we work by induction on the rank
of $x$ in $R(V)$. 

If $x$ has rank $0$ in $R(V)$,
i.e.\ $x$ is terminal in $R(V)$,
then all options of $x$ in $V$
are in $\cP_V$, i.e.\ $\cG((V,y))=0$, which 
gives $\cG((V,x))=1$.

Now let $n>1$. Assume that any vertex with rank less than $n$ in $R(V)$ has finite rank in $V$,
and consider any vertex $x$ with rank $n$ in $R(V)$,
say $\cG_{n}((R(V),x))=m$.

Then using Proposition \ref{prop:usefulSGfacts} again,
\begin{itemize}
    \item[(i)]
    There are $y_0, y_1, \dots, y_{m-1}\in \Gamma_{R(V)}(x)$
    such that for each $i$, 
    $\rank((R(V), y_i))<n$ and
    $\cG((R(V), y_i))=i$.
    Then by the induction hypothesis, 
    $\rank((V, y_i))<\infty$, and part (a) gives
    $\cG((V, y_i))=i+1$.
    \item[(ii)]
    For all $y\in \Gamma_{R(V)}(x)$,
    either $\rank((R(V), y))<n$ and
    $\cG((R(V), y))<m$,
    or there is $z\in \Gamma_{R(V)}(y)$ such
    that $\rank((R(V), z))<n$ and
    $\cG((R(V), z))=m$.
    By the induction hypothesis and part (a) again, 
    then either $\cG(V,y)<m+1$ or there is such a $z$
    with $\cG(V,z)=m+1$.
\end{itemize}
Now consider two possibilities.
Either there is $y\in \Gamma_V(x)$ with $\cG((V,y))=0$.
Then for some large enough $n'$ we get
$\cG_{n'}((V,x))=m+1$, and indeed $x$ has finite rank 
in $V$.
Alternatively, there is no such $y$. Then 
if $x$ had infinite rank in $V$,
we would have $\cG((V,x))=\infty(\cA)$ for 
some $\cA$ with $0\notin \cA$. This would contradict
the assumption that $V$ is draw-free. 
Hence again $x$ must have finite rank in $V$, as required.
\end{proof}

\subsection{Mex labellings, and interpretation 
of $k$-stability in the case of trees}
The material in this section is not used
in the later analysis, but it aims to give
helpful intuition about
the notion of $k$-stability in the case of trees,
showing that it can be interpreted in terms
of consistency of the set of vertices labelled
$0,1,\dots, k$ across all labellings
which locally respect the mex recursions. 

Let $V$ be a locally finite directed graph. 
We call a function $f:V\to\NN$
a \textbf{mex labelling} of $V$
if for all $x\in V$, 
$f(x)=\mex\{f(y), y\in \Gamma_V(x)\}$.

Of course, if $V$ is finite and loop-free,
then there is a unique mex labelling $f$
of $V$ given by $f(x)=\cG((V,x))$ for $x\in V$.

Notice also that any locally finite tree has at least one
mex labelling. To see this we can consider
the sequence of finite graphs $(V_n, n\in\NN)$, 
where $V_n$ is the induced subgraph of $V$
containing all vertices $x$ such that
$\height(x)\leq n$.
Each such $V_n$ is finite and loop-free, and so 
has a mex labelling $f_n$. 
In any mex labelling, the vertex $x$ has value
no greater than the out-degree of $x$ 
(which is finite by assumption). 
Then a compactness/diagonalisation argument
shows that there exists a labelling $f:V\to\NN$
which, on any finite subset $W\subset V$, agrees
with infinitely many of the $f_n$. In 
particular, for any vertex $x$, $f$ agrees with
one of the $f_n$ on $\{x\}\cup\Gamma_V(x)$. Then $f$ obeys the mex recursion
at every such vertex $x$, so $f$ is indeed a mex labelling of $V$.

\begin{proposition}\label{prop:labellings}
Let $V$ be a locally finite tree, and $k\in\NN$.
\begin{itemize}
    \item[(a)]
    Suppose $V$ is $k$-stable. 
Then the set $\{x\in V:f(x)=k\}$ is the same for all mex labellings $f$ of $V$,
and is equal to $\{x\in V:\cG((V,x))=k\}$.
    \item[(b)]
    Suppose $V$ is not $k$-stable,
    but is $(k-1)$-stable. 
    (Ignore the 
    vacuous condition of $(k-1)$-stability for $k=0$.)
    Let $x\in V$ with $\cG(x)=\infty(\cA)$ for
    some $\cA$ not containing $k$. Then 
    there are mex labellings $f$ and $f'$ of $V$ with $f(x)=k$,
    $f'(x)\neq k$.
\end{itemize}
\end{proposition}

Note that the conclusion of part (b) can fail
even for graphs which are acyclic in the sense
of having no directed cycles. 
See Figure \ref{fig:mex-counterexample} for
an example. (The method of proof below makes clear
that the result does extend to bipartite
graphs with no directed cycles.)

\begin{figure}[ht]
\begin{center}
\begin{tikzpicture}[>=Latex]
\node[draw, circle, 
label=north:\textcolor{red}{$1$},
label=south:\textcolor{blue}{$2$}
] (A0) at (0, 0) {$a_0$};

\node[draw, circle, 
label=north:\textcolor{red}{$0$},
label=south:\textcolor{blue}{$0$}
] (B0) at (2, 1) {$b_0$};

\node[draw, circle, 
label=north:\textcolor{red}{$2$},
label=south:\textcolor{blue}{$1$}
] (A1) at (4, 0) {$a_1$};

\node[draw, circle, 
label=north:\textcolor{red}{$0$},
label=south:\textcolor{blue}{$0$}
] (B1) at (6, 1) {$b_1$};

\node[draw, circle, 
label=north:\textcolor{red}{$1$},
label=south:\textcolor{blue}{$2$}
] (A2) at (8, 0) {$a_2$};

\node[draw, circle, 
label=north:\textcolor{red}{$0$},
label=south:\textcolor{blue}{$0$}
] (B2) at (10, 1) {$b_2$};

\node[draw, circle, 
label=north:\textcolor{red}{$2$},
label=south:\textcolor{blue}{$1$},
label=east:.\,.\,.\,.
] (A3) at (12, 0) {$a_3$};

\path[->, thick] 
(A0) edge [bend left=20] (B0)
(A0) edge [bend right=10] (A1)
(B0) edge [bend left=20] (A1)
(A1) edge [bend left=20] (B1)
(A1) edge [bend right=10] (A2)
(B1) edge [bend left=20] (A2)
(A2) edge [bend left=20] (B2)
(A2) edge [bend right=10] (A3)
(B2) edge [bend left=20] (A3)
;

\end{tikzpicture}
\end{center}
    \caption{
    An example showing the conclusion
    of Proposition \ref{prop:labellings}(b)
    can fail even for ``loop-free" graphs 
    (i.e.\ graphs with no directed cycle).
    The directed graph with vertex set 
    $\{a_i, i\in \NN\}\cup\{b_i, i\in\NN\}$,
    and arcs from $a_i$ to $b_i$, from $a_i$ to $a_{i+1}$,
    and from $b_i$ to $b_{i+1}$ for each $i$.
    There are two mex labellings, one shown in red
    above the vertices and the other shown in blue below 
    the vertices. 
    Every position has Sprague-Grundy value $\infty(\emptyset)$, and the graph is not $0$-stable. 
    However, the positions $b_i$ have value
    $0$ in both mex labellings, while the positions $a_i$
    have non-zero values in both mex labellings. 
    \label{fig:mex-counterexample}}
\end{figure}
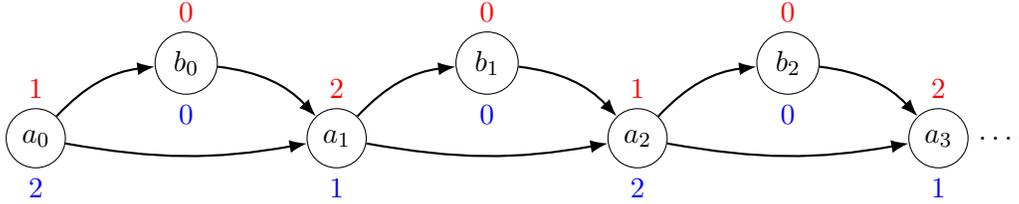

\begin{proof}
We start by proving that if $x\in V$ has finite rank with $\cG(x)=m$, then $f(x)=m$ for all mex labellings $f$ of $V$. (This holds for any locally finite directed graph $V$.)

We proceed by induction on $\rank(x)$. Let $f$ be any mex labelling of $V$. 

If $\rank(x)=0$, then $x$ has no options. Then $\cG(x)=0$, and so
$f(x)=\mex(\emptyset)=0$. 

Now suppose $\rank(x)=n>0$ and $\cG(x)=m$, and that the statement holds for all vertices of rank less than $n$. 

From Proposition \ref{prop:usefulSGfacts}, for each $i$ with $0\leq i\leq m-1$, 
there exists $y_i\in \Gamma(x)$ with $\cG(y_i)=i$ and $\rank(y_i)<n$.
Hence $f(y_i)=i$. 

Also for every $y\in\Gamma(x)$ with $\cG(y)\geq m$, there is $z\in\Gamma(y)$
with $\rank(z)<n$ and $\cG(z)=m$. Then $f(z)=m$, and hence $f(y)\ne m$. 

Thus $x$ has options on which $f$ takes value $0,1,\dots, m-1$, but no option on which $f$ takes value $m$. This gives $f(x)=m$ as required. 

To complete the proof of part (a), suppose that $V$ is $k$-stable, and let $f$ be any mex labelling of $V$. Then any vertex $x$ with infinite rank has $\cG(x)=\infty(\cA)$ for some $\cA$ with $k\in\cA$. Hence there exists $y\in\Gamma(x)$ with $\cG(y)=k$, giving $f(y)=k$. Then $f(x)\ne k$. 
So indeed, the set of vertices $x$ with $f(x)=k$ is exactly the set of $x$ with $\cG(x)=k$. 

We turn to part (b), starting with the case $k=0$. Suppose that $V$ is a locally finite tree
which is not $0$-stable. Let $x$ be any vertex with $\cG(x)=\infty(\cA)$
for some $\cA$ not containing $0$ (that is, $x\in\cD$). 

Take any $n\geq \height(x)$. 
Since the game from position $x$ is drawn,
if we consider the game on the truncated graph 
$V_n$ described just before the statement of the
proposition, so that all vertices at height $n$ become terminal, then position $x$ becomes a first-player win 
if $n-\height(x)$ is odd, and a second-player win
if $n-\height(x)$ is even. 

Then we can apply again the compactness argument mentioned 
before the statement of Proposition \ref{prop:labellings}, separately
for odd $n$ and even $n$. This yields
two mex labellings $f$ and $f'$, one of which gives
value $0$ to $x$, and the other of which gives
a strictly positive value to $x$, as required. This completes the proof of part (b) in the case $k=0$. 

Now we extend to $k>0$. Suppose $V$ is $(k-1)$-stable but not $k$-stable. 
As in Corollary \ref{corr:Vk}, 
we can apply the reduction operator $k-1$ times, 
removing all the vertices $y\in V$ with $\cG((V,y))<k$,
to arrive at the graph $R^k(V)$.

Any $v\in\R^k(V)$ either has $\cG((V,x))=m$ for some finite $m\geq k$, 
or $\cG((V,x))=\infty(\cA)$ for some $\cA$ with $\{0,\dots,k-1\}\subseteq\cA$. 
It is then easy to check that whenever $\hat{f}: R^k(V)\mapsto \NN$ is a mex labelling of $R^k(V)$,
we can obtain a mex labelling $f:V\mapsto \NN$ of $V$ by defining
\begin{equation}\label{fhattof}
f(x)=\begin{cases}
\cG((V,x)),&\text{if }\cG((V,x))<k\\
\hat{f}(x)+k,&\text{otherwise}.
\end{cases}
\end{equation}

Let $x\in V$ with $\cG(V,x)=\infty(\cA)$ for some $\cA$ containing $0,\dots, k-1$ but not $k$. 
Then, by applying Theorem \ref{thm:reduction} $k$ times, we have $x\in R^k(V)$ and $\cG(R^k(V), x))=\infty(\cB)$ where $\cB=\cA-k$. In particular, 
$0\notin\cB$ (that is, the position $x$ in $R^k(V)$ is a draw). We wish to show that there are mex labellings $f$, $f'$ of $V$ such that
$f(x)=k$ and $f'(x)\ne k$. In light of (\ref{fhattof}), it is enough to show that there are 
mex labellings $\hat{f}, \hat{f}'$ of $R^k(V)$ such that $\hat{f}(x)=0$ and $\hat{f}'(x)>0$. 

Since $x$ is a draw in $R^k(V)$, we would like to use the same approach as in the $k=0$ case. The situation is more complicated since the graph $R^k(V)$ may not be connected. However, 
the graph $R^k(V)$ is a union of finitely or countably many disjoint trees. Any labelling
which restricts to a mex labelling of each tree component is a mex labelling of the whole graph. 
So it suffices to find mex labellings of the tree component of $R^k(V)$ which contains $x$,
one of which assigns value $0$ to $x$ and another of which assigns strictly positive value
to $x$. This indeed can be done using the same compactness argument used in the $k=0$ case.

This completes the proof of part (b). 
\end{proof}

\section{Random game-trees}\label{sec:GW}

\subsection{\BGW{} trees}
A \textbf{\BGW{}} (or \textbf{Bienaym{\'e}}) \textbf{branching process} is constructed as follows. We fix some 
\textbf{offspring distribution} which is a probability distribution $\bp=(p_k, k\in\NN)$
on the non-negative integers. The process begins with a single individual, called
the root. The root individual has a random number of children, distributed
according to the offspring distribution, which form generation $1$. Then each of 
the members of generation $1$  has a number of children according to the offspring distribution,
forming generation $2$, and so on. All family sizes are independent. 
See for example \cite{GrimmettWelshbook} for a basic introduction, 
and \cite{LeGalltreesurvey} for much more depth including a rigorous construction. 

We derive a directed graph from the process by regarding each individual as a vertex,
and putting an arc to each child from its parent. In this way
each vertex of the graph has in-degree $1$, except for the root which has in-degree $0$.
We call the resulting graph a \textbf{\BGW{} tree}. This tree has a natural
self-similarity property: conditional on the number of the children of the root
being $k$, the subtrees rooted at those children are independent and each one has
the distribution of the original \BGW{} tree.

We assume always that $p_0>0$, so that the tree can have terminal vertices. 

A key role in what follows will be played by the \textbf{probability generating function}
of the offspring distribution, defined by
\[
\phi(s)=\sum_{k\geq 0}p_k s^k.
\]
The function $\phi$ is strictly increasing on the interval
$[0,1]$, and maps $[0,1]$ bijectively to the interval $[p_0, 1]$.

A fundamental result is a criterion for the tree to be infinite, in terms 
of the mean $\mu=\sum_{k\geq 0} k p_k = \phi'(1)$ of the offspring distribution $\bp$.
Excluding the trivial case $p_1=1$
(where with probability $1$ the tree consists of a 
single path) one has that whenever $\mu\leq 1$, the tree
is finite with probability $1$, and whenever $\mu>1$, there is positive probability for the tree to be infinite. 

If $d=\sup\{k:p_k>0\}$ is finite, we say the 
offspring distribution has 
\textbf{maximum out-degree $d$}. Otherwise we say 
that the offspring distribution has 
unbounded vertex degrees. 

\subsection{\BGW{} games}\label{sec:GWgames1st}
We will consider \textbf{\BGW{} games},
i.e.\ games whose directed graph is a \BGW{} tree $T$.

We start with a very simple lemma which helps 
simplify the language:

\begin{lemma}\label{lem:GWequivalences}
Consider a \BGW{} tree $T$, with root $o$.
Let $\cC$ be any set of possible Sprague-Grundy values. 
The following are equivalent:
\begin{itemize}
    \item[(a)]
    $\PP(\cG((T,o))\in\cC)>0$;
    \item[(b)]
    $\PP(\cG((T, u))\in\cC \text{ for some }u\in T)>0$.
    \end{itemize}
\end{lemma}

For example, the tree $T$ is draw-free with probability $1$
iff the probability that the root is drawn is $0$.
So we do not need to distinguish carefully between
saying that ``the tree has draws with positive probability"
and that ``the root is drawn with positive probability".
More generally, the tree $T$ is $k$-stable with probability $1$ iff the probability that $\cG(T,o)=\infty(\cA)$ for some $\cA$ not containing $\{0,1,\dots,k\}$ is $0$. 

\begin{proof}[Proof of Lemma \ref{lem:GWequivalences}]
Trivially (a) implies (b). 
On the other hand, if (a) fails, so that
that $\PP(\cG(T,o))\in\cC=0$, then the self-similarity
of the \BGW{} tree, the fact that the tree
has at most countably many vertices, and the countable
additivity of probability measures, combine to give 
that $\PP(\cG((T, u))\in\cC \text{ for some }u\in T)=0$
also, so that (b) also fails. 

\end{proof}

The question of when a \BGW{}
game has positive probability to be a draw
was considered in \cite{GaltonWatsongames}.

Let $\cP_n$ be the set of vertices from which the second player has a winning strategy
that guarantees to win within $2n$ moves ($n$ by each player), and let $P_n$ be the 
probability that $o\in\cP_n$. Note that $o\in\cP_n$ iff for every child $u$ 
of $o$, $u$ itself has a child in $\cP_{n-1}$. This leads to the following recursion
for the probabilities $P_n$ in terms of the generating function:
\begin{equation}\label{Pnrecursion}
P_n=1-\phi(1-\phi(P_{n-1})).
\end{equation}
Now let $P$ be the probability that $o\in\cP$. We have $P=\lim_{n\to\infty}P_n$. 
Taking limits in (\ref{Pnrecursion}), and using the fact
that the generating function $\phi$ is continuous and increasing on $[0,1]$,
we obtain part (a) of the following result. 
A similar approach involving the probability of winning strategies for 
the first player gives part (b). For full details, see \cite{GaltonWatsongames}.

\begin{proposition}[Theorem 1 of \cite{GaltonWatsongames}]\label{prop:draws}
Define a function $h:[0,1]\to[0,1]$ by 
\begin{equation}\label{hdef}
h(s)=1-\phi\big(1-\phi(s)\big).
\end{equation}
\begin{itemize}
    \item[(a)] 
    $P:=\PP((T,o)\in\cP)$ is the smallest fixed point of $h$ in $[0,1]$.
    \item[(b)]
    If $N:=\PP((T,o)\in\cN)$, then $1-N$ is the largest fixed point of $h$ in $[0,1]$.
\end{itemize}
\end{proposition}

\begin{corollary}
\label{corr:draw-criterion}
    $D:=\PP((T,o)\in\cD)=1-N-P$ is positive iff the function $h$ defined by (\ref{hdef}) has more than one fixed point in $[0,1]$.
\end{corollary}

Note that $h$ defined in (\ref{hdef}) is the second
iteration of the function $1-\phi$. The function $1-\phi$
is continuous and strictly decreasing, mapping $[0,1]$
to $[1-p_0,0]$. It follows that $1-\phi$ has precisely
one fixed point in $[0,1]$, and that fixed point is 
also a fixed point of $h$. So Corollary  \ref{corr:draw-criterion} tells us that the game has positive
probability of draws iff $h$ has further fixed points
which are not fixed points of $1-\phi$.

Two particular families of offspring distributions had been
considered earlier. The Binomial$(2,p)$ case was studied
by Holroyd in \cite{Ander-thesis}. The case of 
the Poisson offspring family is closely related 
to the analysis of the \textit{Karp-Sipser algorithm}
used to find large matchings or independent sets of
a graph, which was introduced by Karp and Sipser in 
\cite{KarpSipser}; the link to games is not described
explicitly in that paper, but the choice of notation
and terminology makes clear that the authors were aware of it.

One particular focus of \cite{GaltonWatsongames}
was on the nature of the phase transitions 
between the set of offspring distributions
without draws, and the set of offspring distributions
with positive probability of draws. This transition 
can be either continuous or discontinuous. 
Without going into precise details, we illustrate
with a couple of examples. 

\begin{example}[Poisson distribution -- continuous phase
transition]\label{ex:Poissondraw}
The Poisson($\lambda$) offspring family was considered in 
Proposition 3.2 of \cite{GaltonWatsongames}. 
The game has probability $0$ of a draw if $\lambda\leq e$,
and positive probability of a draw if $\lambda>e$. 
The phase transition is illustrated in Figure
\ref{fig:Poissondraw}. For $\lambda\leq e$, 
the function $h$ has only one fixed point, 
while for $\lambda>e$, $h$ has three fixed points. 
The additional fixed points emerge continuously
from the original fixed point as $\lambda$ goes above $e$.
Note that the probability of a draw at the critical point 
itself is $0$; more strongly, we have the draw probability $\PP(o\in\cD)$ is a continuous function of $\lambda$.
\end{example}

\begin{figure}
\begin{center}
\includegraphics[width=0.45\linewidth]{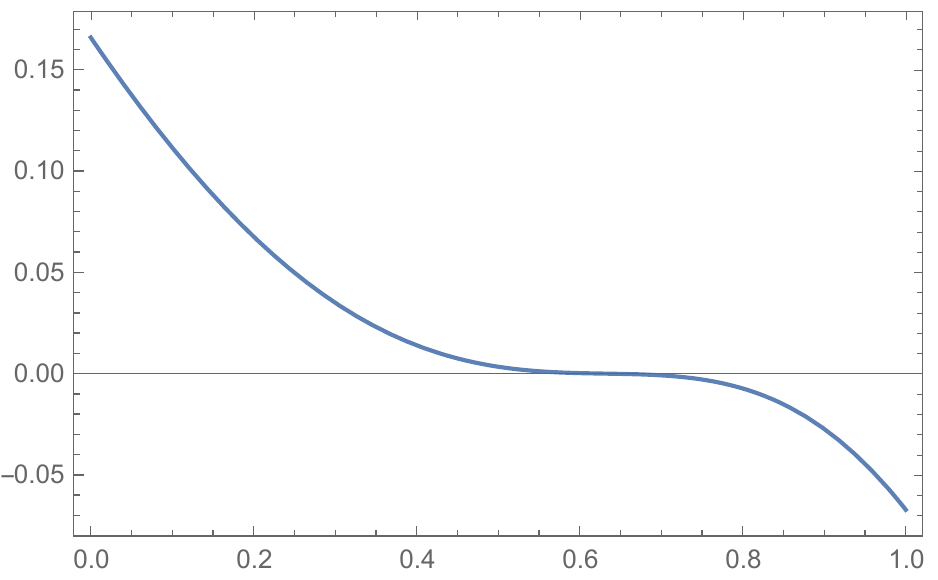}
\includegraphics[width=0.45\linewidth]{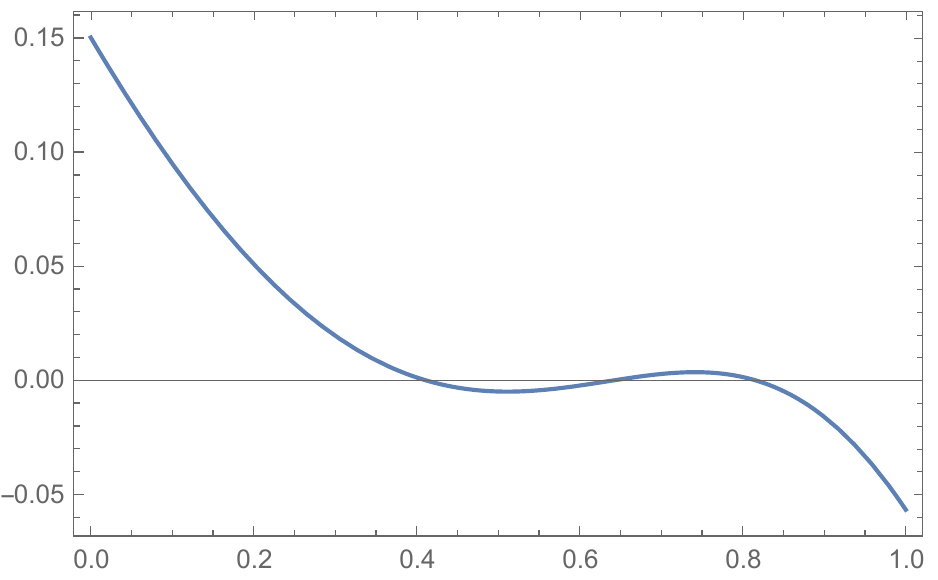}
\end{center}
\caption{
\label{fig:Poissondraw}
An illustration of the phase transition 
from the non-draw to the draw region, 
for Poisson($\lambda)$ offspring distributions
(see Example \ref{ex:Poissondraw}).
The two plots show the function $h(s)-s$ for $s\in[0,1]$,
where $h$ is defined by (\ref{hdef}). 
The fixed points of $h$ are those $s$
where $h(s)-s$ crossing the horizontal axis.
On the left, $\lambda=2.7$, just below the critical 
point $\lambda=e$; the function $h$ has just one 
fixed point. On the right, $\lambda=2.8$, just 
above the critical point; now $h$ has three fixed points. 
}
\vspace{1cm}
    \centering
    \includegraphics[width=0.45\linewidth]{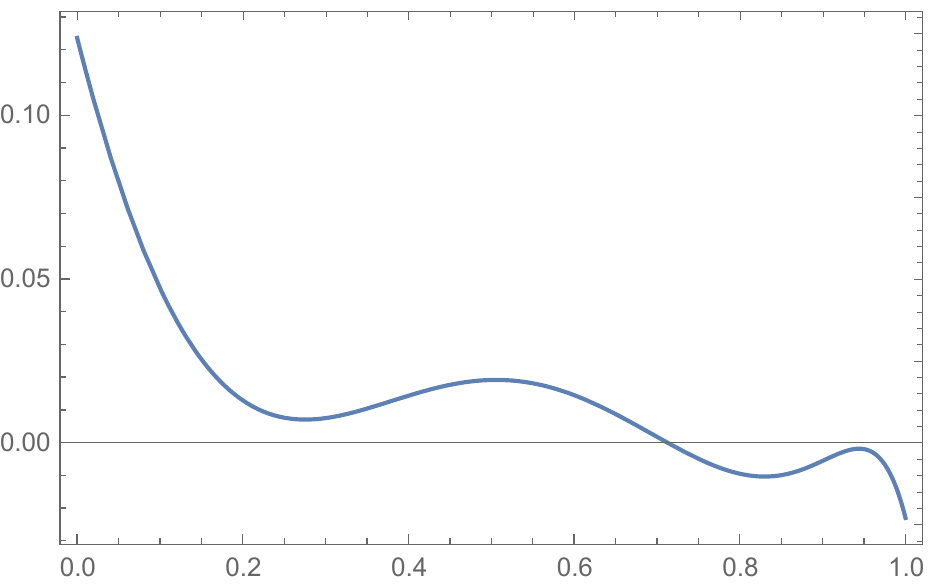}
\includegraphics[width=0.45\linewidth]{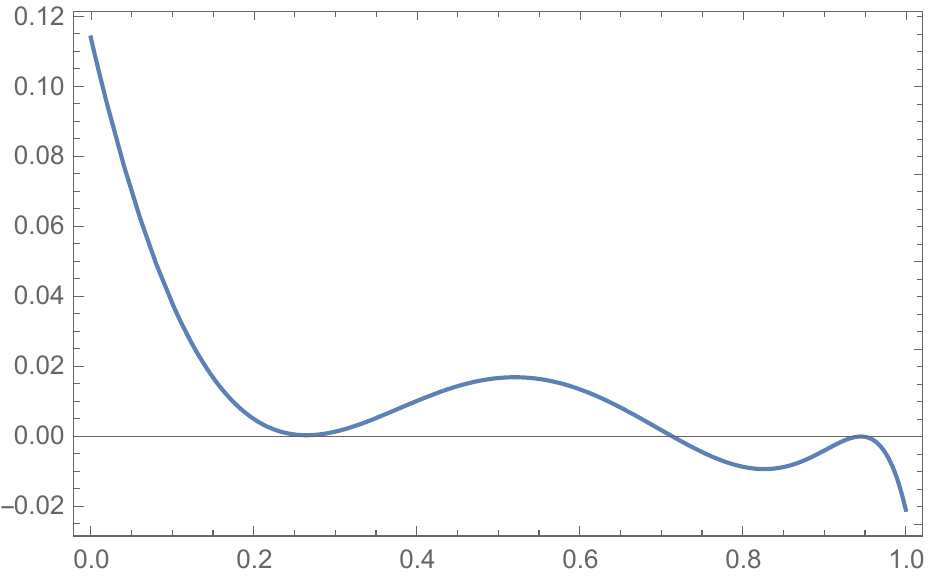}
    \caption{
The phase transition for the
family of offspring distributions given in
Example \ref{ex:discontinuous}, with
$p_0=1-a, p_2=a/2, p_{10}=a/2$. 
Again the function $h(s)-s$ is shown for $s\in[0,1]$.
On the left $a=0.977$, and on the right $a=0.979\approx a_c$.
Unlike in Figure \ref{fig:Poissondraw}, 
at the critical point there are already 
multiple fixed points of $h$; at $a_c$, the draw probability 
jumps from $0$ to a positive value around 
$0.681$, which is the distance between the 
minimum and maximum fixed points of $h$. 
\label{fig:discontinuous}}
\end{figure}

\begin{example}[A discontinuous phase transition]
\label{ex:discontinuous}
Consider a family of offspring distributions with
$p_0=1-a$, $p_2=a/2$, $p_{10}=a/2$, where $a\in[0,1]$. 
This family is used in the proof of Proposition 4(i) of 
\cite{GaltonWatsongames}. Again there is some critical 
point $a_c\approx 0.979$ such that there is positive probability of a draw for $a>a_c$ and not for $a<a_c$.
However, unlike in the Poisson case above, at the critical 
point itself, the function $h$ already has three fixed points,
and the probability $\PP(o\in\cD)$ jumps discontinuously 
from $0$ for $a<a_c$ to approximately $0.61$ at $a=a_c$
itself. The difference in the nature of the emergence of
the additional fixed points of $h$ can be 
seen by comparing Figure~\ref{fig:Poissondraw} and Figure~\ref{fig:discontinuous}.
\end{example}

\subsection{Existence of infinite rank vertices in \BGW{} games}

Now we go beyond the question of whether 
draws have positive probability, to ask 
more generally about the extended Sprague-Grundy values that occur in a \BGW{} game. 
A specific question will be whether, 
when draws are absent, there are still 
some infinite rank positions. 
As suggested by Corollary \ref{corr:Vk},
we can investigate the $k$-stability 
of the tree $T$ by looking 
at whether draws occur for the reduced trees $R^k(T)$.
The reduction operator behaves particularly
nicely in the setting of a \BGW{} tree.

\begin{theorem}\label{thm:phi1}
Consider a \BGW{} tree $T$ whose
offspring distribution $(p_n, n\geq 0)$ 
has probability
generating function $\phi$.

Suppose the tree is draw-free with probability 
$1$. Let $P$ be the probability that the root 
$o$ is a $\cP$-position. 

Condition on the event $\OO(o)=\cN$, 
and consider the graph obtained by removing all the $\cP$-positions. 
Let $T^{(1)}$ denote the component connected 
to the root $o$ in this graph. 
Then $T^{(1)}$ is again a \BGW{} tree
rooted at $o$, whose 
offspring distribution
has probability generating function 
given by
\begin{equation}\label{phi1}
\phi^{(1)}(s)= \frac{1}{1-P} 
\left[
\phi\big(P+s(1-P)\big) - \phi\big(s(1-P)\big)
\right].
\end{equation}
\end{theorem}

\begin{proof}
Since we assume that $T$ has no draws, each vertex of $T$ is either a $\cP$-position or an $\cN$-position. 
The type of a vertex is determined by the subtree
rooted at that vertex. Conditional on the number
of children of the root, the subtrees rooted at each child are independent and each have the same
distribution as the original tree. In particular,
each child is independently
a $\cP$-position with probability $P$, and an $\cN$-position with probability $1-P$.

This gives us a two-type \BGW{} process.
We have the familiar recursion that a vertex is an $\cN$-position iff at least one of its children is a $\cP$-position. 

We condition on the root being of type $\cN$, 
and retain only its $\cN$-type children,
and the $\cN$-type children of those children, 
and so on. This gives a one-type \BGW{} process,
and its offspring distribution is the 
distribution of the number of $\cN$-type children
of the root in the original process, conditional 
on the root having type $\cN$.

The probability that the root has $k$ $\cP$-children
and $m$ $\cN$-children is 
\[
p_{m+k} \Ch{m+k}{k} P^k (1-P)^m.
\]
We can sum over $k\geq 1$ to obtain
the probability that the root
is of type $\cN$ and has $m$ $\cN$-children.
Finally, we can condition on the event that the root has type 
$\cN$ (which has probability $1-P$), to obtain 
that the conditional probability that the root has $m$ $\cN$-children given that it has type $\cN$ is
\[
p^{(1)}_m:=\frac{1}{1-P}\sum_{k=1}^\infty p_{m+k} \Ch{m+k}{k} P^{k}(1-P)^m.
\]
Finally we want to calculate the probability
generating function 
$\phi^{(1)}(s):=\sum_{m\geq 0} s^m p^{(1)}_m$
of this distribution. This can easily be
done using the binomial theorem to arrive
at the form given in (\ref{phi1}).
\end{proof}

Combining Corollary \ref{corr:Vk} and Theorem \ref{thm:phi1} is the key to 
studying the infinite-rank vertices
of our \BGW{} tree $T$; 
see the strategy described at the beginning of 
Section \ref{sec:examples}.

We finish this section with a result about the possible infinite Sprague-Grundy values that can occur in a \BGW{} game. 
Essentially, the value $\infty(\cA)$
has positive probability to appear for every finite $\cA$
which is not ruled out either by $k$-stability or by
finite maximum vertex degree.
Most notably, part (a)(i) says that 
for a tree which has draws and for which the offspring
distribution has infinite support, \textit{all}
finite $\cA$ have positive probability.

\begin{proposition}\label{prop:possiblevalues}
Consider the game on a \BGW{} tree. 
\begin{itemize}
\item[(a)]
Suppose there is positive probability of a draw. 
\begin{itemize}
    \item[(i)]
If the vertex degrees are unbounded, 
then for any finite $\cA\subset \NN$, 
there is positive probability that $\cG(o)=\infty(\cA)$.
\item[(ii)]
If the maximum out-degree is $d$, then 
there is positive probability that $\cG(o)=\infty(\cA)$
iff $\cA\subset\{0,1,\dots,d\}$ with $|\cA|\leq d-1$.
\end{itemize}
\item[(b)]
For $k\geq 1$, suppose that the tree
is $(k-1)$-stable with probability $1$, but has
positive probability not to be $k$-stable. 
\begin{itemize}
    \item[(i)]
If the vertex degrees are unbounded, 
then for any finite $\cA\subset \NN$
which contains $\{0,\dots, k-1\}$, 
there is positive probability that $\cG(o)=\infty(\cA)$.
\item[(ii)]
If the maximum out-degree is $d$, then 
there is positive probability that $\cG(o)=\infty(\cA)$
iff $\{0,1,\dots, k-1\}\subseteq\cA\subset\{0,1,\dots,d\}$ with $|\cA|\leq d-1$.
\end{itemize}
\end{itemize}
\end{proposition}

\begin{proof}
First we note that all finite Sprague-Grundy values 
have positive probability, up to the maximum out-degree $d$ if there is one. This is easy by induction. We know that value $0$ is possible since any terminal position has value $0$.
If values $0,1,\dots, k-1$ are possible, and 
it is possible for the root to have degree $k$ or larger, 
then there is positive probability that the set of values of
the children of the root is precisely $\{0,1,\dots, k-1\}$,
giving value $k$ to the root as required. 

Now for part (a), since draws are possible, 
the value $\infty(\cB)$ has positive probability for 
some $\cB$ not containing $0$. In that case, there is
positive probability
for all the children of the root to have value $\infty(\cB)$,
and then the root has value $\infty(\emptyset)$.

So the value $\infty(\emptyset)$ has positive probability. 
Now if $\cA$ is any finite set such that the number
of children of the root can be as large as $|\cA|+1$,
then there is positive probability that the set of values of the children of the root is precisely $\cA\union\{\infty(\emptyset)\}$, and in that case 
the value of the root is $\infty(\cA)$ as required. 

Finally, if $|\cA|$ is greater than or equal to the maximum
degree, then the value $\infty(\cA)$ is impossible, 
since any vertex with such a value must have
at least one child with value $m$ for each $m\in\cA$, and 
additionally at least one child with infinite rank.

We can derive the result for part (b) by 
applying part (a) to the \BGW{} tree $T^{(k)}$
obtained by conditioning the root to have 
Sprague-Grundy value not in $\{0,1,\dots, k-1\}$, 
and removing all the vertices with values $\{0,1,\dots, k-1\}$
from the graph, as described above. Theorem 
\ref{thm:reduction} tells us that if 
the resulting tree has positive probability to have a node
with value $\infty(\cA)$, then the original tree has 
positive probability to have a node with value 
$\infty(\cB)$ where $\cB=\{b\geq k: b-k\in \cA\}\cup\{0,1,\dots, k-1\}$, and the desired results follow. 
\end{proof}

\begin{remark}\label{rem:allpossible}
Suppose we have a \BGW{} tree
$T$ with positive probability to be infinite, 
and a set $\cC$ of Sprague-Grundy values with
$\PP(\cG(o)\in\cC)>0$. A straightforward extension
of Lemma \ref{lem:GWequivalences} says that 
conditional on $T$ being infinite, with probability $1$
there exists $u\in T$ with $\cG(u)\in\cC$.

Combining with Proposition \ref{prop:possiblevalues},
we get the following appealing property. 
If $T$ has unbounded vertex degrees, and 
positive probability of draws, then conditional 
on $T$ being infinite, with probability $1$,
vertices with every possible 
extended Sprague-Grundy value are found in the tree.
\end{remark}

\section{Examples}\label{sec:examples}
First we lay out how to use the results
of the previous sections to address
the question of which 
infinite-rank Sprague-Grundy values
have positive probability for a
given \BGW{} tree $T$.

Let $\phi$ be the probability generating function
of the offspring distribution of $T$.
To examine whether $T$ can have draws, 
we apply the criterion given 
Corollary \ref{corr:draw-criterion}:
$T$ is draw-free with probability $1$ iff the function
$h(s)=1-\phi(1-\phi(s))$ has a unique fixed point.

If so, we use the procedure in Theorem \ref{thm:phi1}.
We condition the root to be a $\cN$-position,
we remove all the $\cP$-positions, and we retain
the connected component of the root, 
to obtain a new \BGW{} tree $T^{(1)}$
with an offspring distribution 
whose probability generating function is $\phi^{(1)}$.
We then examine whether or not this new generating function 
gives a draw-free tree.

If it does, we can repeat the procedure again,
producing a new generating function which we call
$\phi^{(2)}$, corresponding to removing the 
positions with Sprague-Grundy values $0$ and $1$
from the original tree. 

If we perform $k$ reductions and still 
have a draw-free tree at every step, this tells
us that our original tree was $k$-stable with 
probability $1$. 

If the iteration of this procedure never produces
a tree with positive probability of a draw, then
the original tree had probability $0$ of having 
infinite-rank vertices. 
(Note that for example if at any step 
we arrive at a tree which is sub-critical,
i.e.\ whose offspring distribution has mean less
than or equal to $1$ and which therefore has
probability $1$ to be of finite size, 
then we know that every further reduction 
must give rise to a draw-free tree.)

We now apply this strategy to  
a few different examples
of families of offspring distributions. We see
a surprising range of types of behaviour.

\begin{example}[Poisson case, continued]
\label{ex:Poisson}
\BGW{} trees with Poisson offspring distribution behave particularly nicely under the graph reduction operation. 
This allows us to give a complete analysis of
the Poisson case without any need for calcuations or
numerical approximation. 

The tree has positive probability to be infinite 
iff $\lambda>1$. We already saw in 
Example \ref{ex:Poissondraw} that
there is positive probability of a draw iff $\lambda>e$.

Suppose we are in the $\lambda\leq e$ case without draws.
So each node is a $\cP$-node (with probability $P$) 
or a $\cN$-node (with probability $1-P$).

By basic properties of the Poisson distribution, the
number of $\cP$-children of the root is Poisson$(\lambda P)$-distributed,
and the number of $\cN$-children of the root is Poisson$(\lambda(1-P))$-distributed,
and the two are independent.

If we condition the root to have at least one $\cP$-child, and then remove all its $\cP$-children, then because of the independence of the number of $\cP$-children and the number of
$\cN$-children, we are simply left with a Poisson($\lambda(1-P)$) number of children. 

So we again have a Poisson \BGW{} tree, but now with
a new parameter $\lambda^{(1)}<\lambda$.
Since $\lambda^{(1)}<e$, the new tree is still draw-free with probability $1$.

Hence, to adapt the terminology of the introduction,
in the Poisson case we may see a ``blatantly infinite" game once $\lambda>e$, 
but for $\lambda\leq e$ we are at worst ``latently infinite". There is no $\lambda$ which gives ``patently infinite"
behaviour whereby draws are absent
but infinite rank vertices have positive probability.
\end{example}

\begin{example}[Degrees $0$ and $4$]
\label{ex:04}
We return to the example in the introduction,
where all outdegrees are $0$ or $4$. 
We have $p_4=p$ and $p_0=1-p$ for some $p\in(0,1)$.

If $p\leq a_0:=1/4$ then the mean offspring size is 
less than or equal to $1$, and the tree is finite with 
probability $1$. 

One can show algebraically that 
there is positive probability of a draw
iff $p>a_2:=5^{3/4}\approx 0.83593$.
Namely one can obtain that the function $h$
defined in (\ref{hdef}) has derivative 
less than $1$ on $[0,1]$ for all $p\leq a_2$
(except for a single point in the case $p=a_2$),
and so $r$ has just one fixed point for such $p$. 
Meanwhile for $p>a_2$ there is a fixed point $s^*$ 
of the function $1-\phi$ for which $h'(s^*)>1$, and this 
can be used to show that $r$ has at least two 
further fixed points. Corollary \ref{corr:draw-criterion}
then gives the result. 

Between $a_0$ and $a_2$ there exist no draws, but the
tree is infinite with positive probability, so we may 
ask whether there can exist positions with infinite rank. 

Numerically, we observe a phase transition around
the point $a_1\approx 0.52198$. For $p\leq a_1$, 
we know that the tree $T$ has zero probability of a draw,
and we observe that the same is also true for the trees
$T^{(1)}$ and $T^{(2)}$ (their maximum out-degrees are
$3$ and $2$ respectively, so their generating functions
$\phi^{(1)}$ and $\phi^{(2)}$
are cubic and quadratic respectively. The tree $T^{(3)}$ 
has vertices of out-degrees only $0$ and $1$, and 
will also be finite with probability $1$, so we do not need
to examine $T^{(k)}$ for any higher $k$.)

Hence for $p\in(a_0,a_1]$, we have the ``latently infinite"
phase where all Sprague-Grundy values are finite with 
probability $1$.

However, for $p\in(a_1,a_2]$ we observe that
the function $h^{(1)}(s):=1-\phi^{(1)}(1-\phi^{(1)}(s))$
has more than one fixed point. Consequently, 
there is positive probability of a draw in the tree
$T^{(1)}$. The tree $T$ has positive probability not to be
$1$-stable, and so to have positions of infinite rank. 

The behaviour of $h$, $h^{(1)}$ and $h^{(2)}$ 
around the phase transition point $p=a_1$ is 
shown in Figure \ref{fig:04graphs}.
Although the precise nature and location of this
phase transition is only found
numerically, it is not hard to show rigorously
that for $p$ just above  $a_0$, 
the functions $h^{(1)}$ and $h^{(2)}$ have only
one fixed point, while for $p$ just below $a_2$, 
the function $h^{(1)}$ has more than one fixed point,
so that the family of distributions does display
all four of the 
``finite", ``latently infinite", ``patently infinite"
and ``blatantly infinite" types of behaviour.

\begin{figure}
    \centering
        \includegraphics[width=0.32\linewidth]{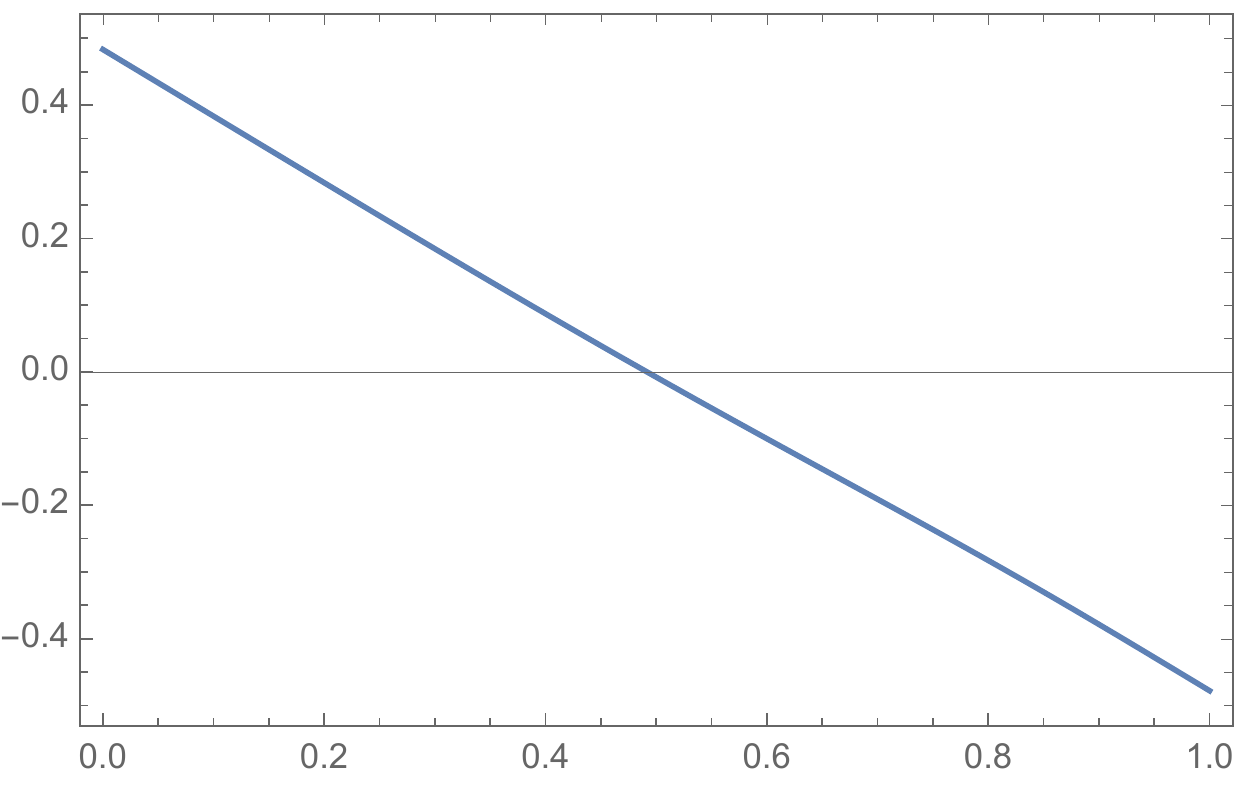}
        \includegraphics[width=0.32\linewidth]{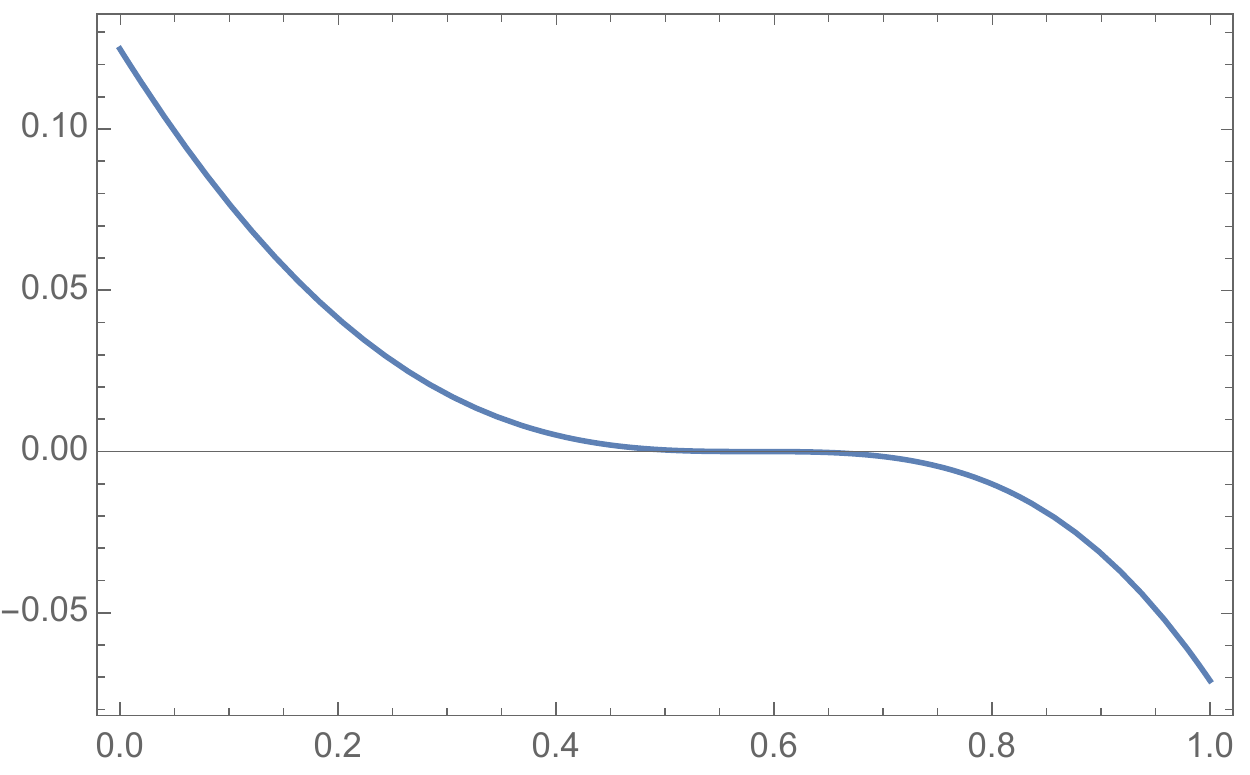} \includegraphics[width=0.32\linewidth]{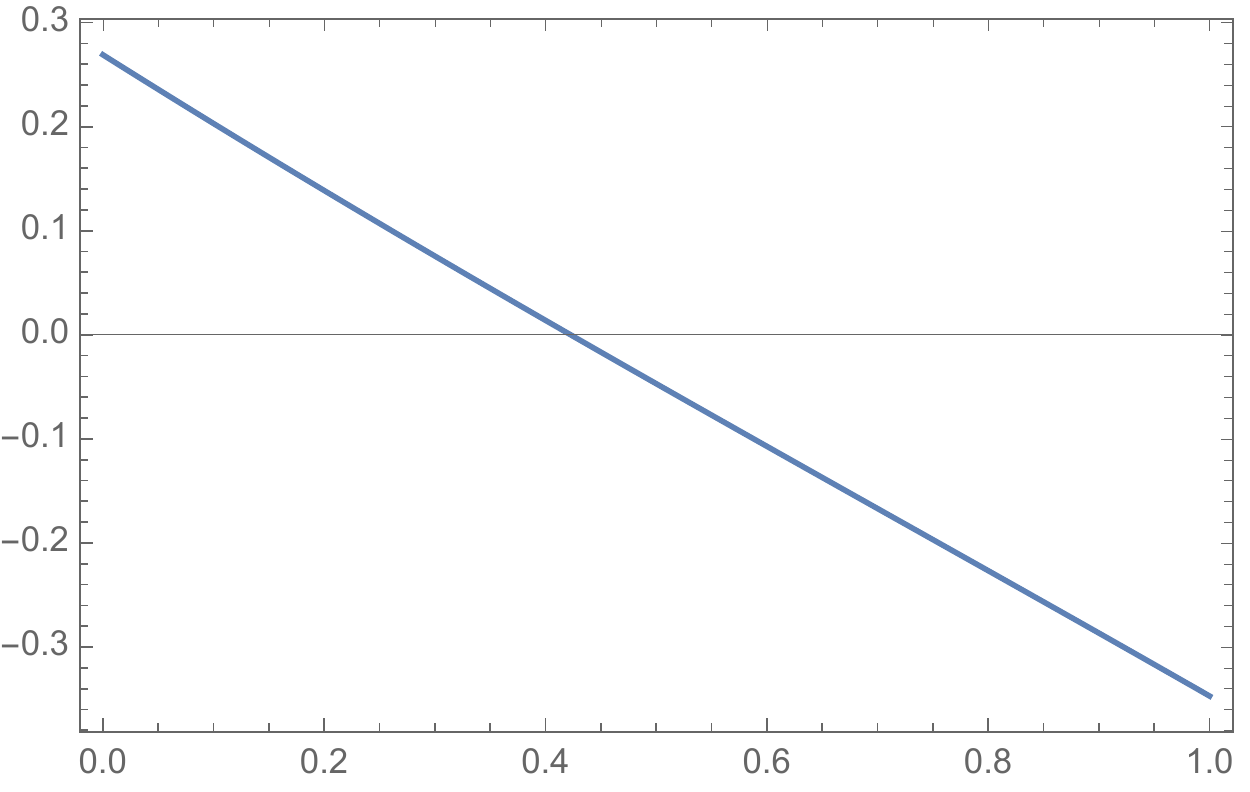}
    \caption{The case of the $0$-or-$4$ distribution from 
    Example \ref{ex:04}, with $p=0.52198\approx a_1$.
    From left to right the three graphs show the functions
    $h(s)-s$, $h^{(1)}(s)-s$, and $h^{(2)}(s)-s$.
    As $p$ moves through the critical point $a_1$, 
    the function $h^{(1)}$ acquires multiple fixed points. 
    For $p\leq a_1$, the tree has only finite-rank vertices. 
    For $p>a_1$, the tree no longer has probability $1$ to 
    be $1$-stable, and for example the Sprague-Grundy value $\infty({0})$ has positive probability. 
    \label{fig:04graphs}}
\end{figure}
\end{example}

\begin{example}[Geometric case]\label{ex:geometric}
We now consider the family of geometric offspring
distributions, with $p_k=q^k(1-q)$ for
$k=0,1,2,\dots$, for some $q\in(0,1)$. 

Rather surprisingly, there is no $q$ for which 
draws have positive probability! See for example
Proposition 3(iii) of \cite{GaltonWatsongames}. 
(This shows for example that
the property of having positive probability of draws
is not monotone in the offspring distribution. 
If we take any $\lambda>e$, then as discussed above, the Poisson($\lambda$)
distribution has positive probability of draws, 
but for $q$ sufficiently large, this distribution
is stochastically dominated by a Geometric($q$)
distribution, which does not have draws.)

However, other interesting
phase transitions for the geometric family
do occur. Numerically, we observe that there are 
critical values
$q_0=1/2, q_1\approx 0.88578, q_2\approx 0.88956, q_3\approx 0.923077$ such that:
\begin{itemize}
    \item For $q\leq 0.5$, the tree is finite with probability $1$.
    \item For $q\in(0.5, q_1]$, there are infinite paths with positive probability,
    but the tree is $3$-stable with probability $1$.
    In fact for $q$ sufficiently close to $0.5$,
    the tree $T^{(1)}$ is finite with probability $1$,
    and so in fact the tree is $k$-stable for all $k$, i.e.\ all positions have finite rank 
    (the latently infinite phase). It seems plausible
    that in fact the latently infinite phase continues
    all the way to $q_1$, but we do not know how to 
    demonstrate that.
    \item For $q\in(q_1, q_2]$, 
    with positive probability the tree is not $3$-stable; 
    however it continues to be $2$-stable.
    \item For $q\in(q_2, q_3]$, 
    with positive probability the tree is not $2$-stable;
    however it continues to be $1$-stable.
    \item For $q\geq q_3$, 
    with positive probability the tree is not $1$-stable
    (but as we know, it continues to be $0$-stable, 
    i.e.\ draw-free, for all $q$).
    \end{itemize}

\begin{figure}
    \centering
\includegraphics[width=0.32\linewidth]{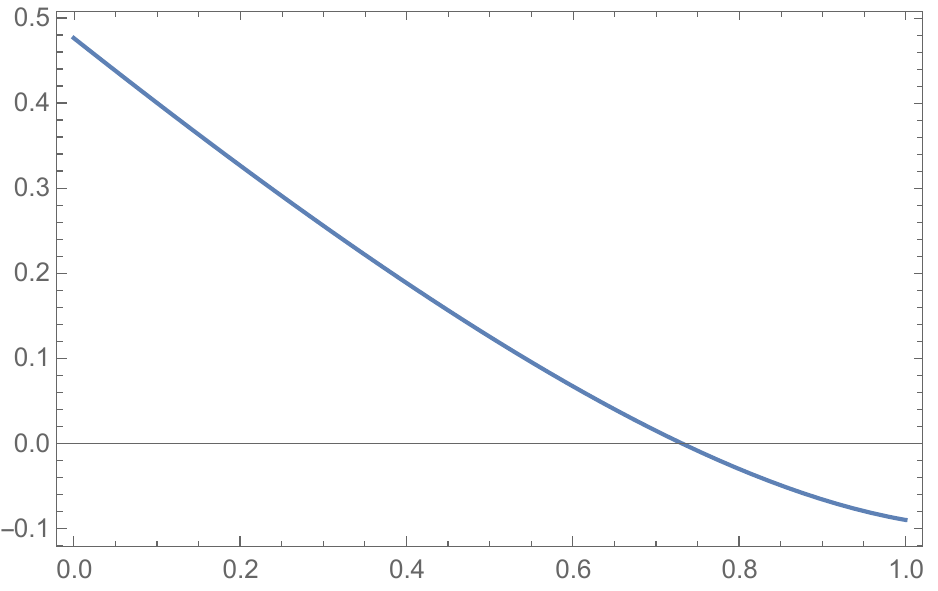}
\includegraphics[width=0.32\linewidth]{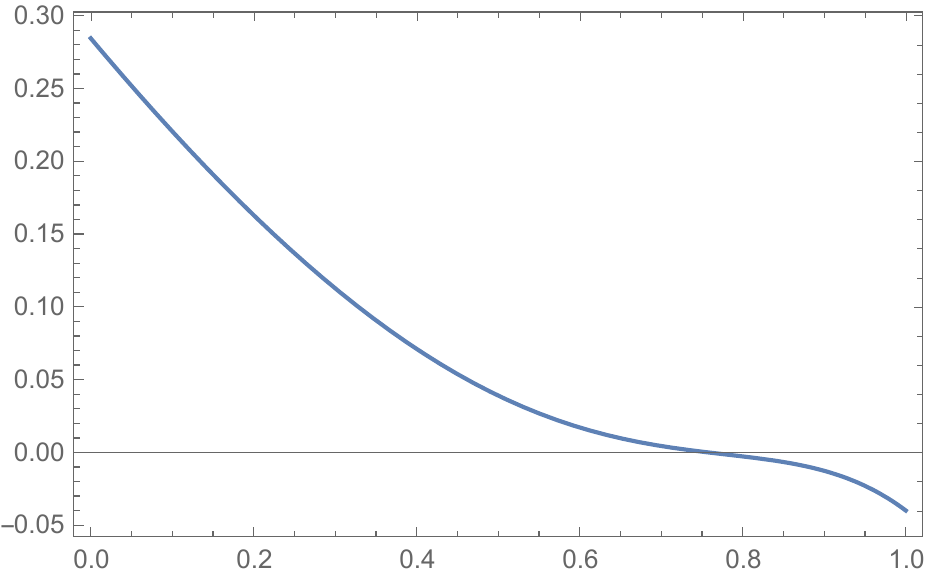}
\includegraphics[width=0.32\linewidth]{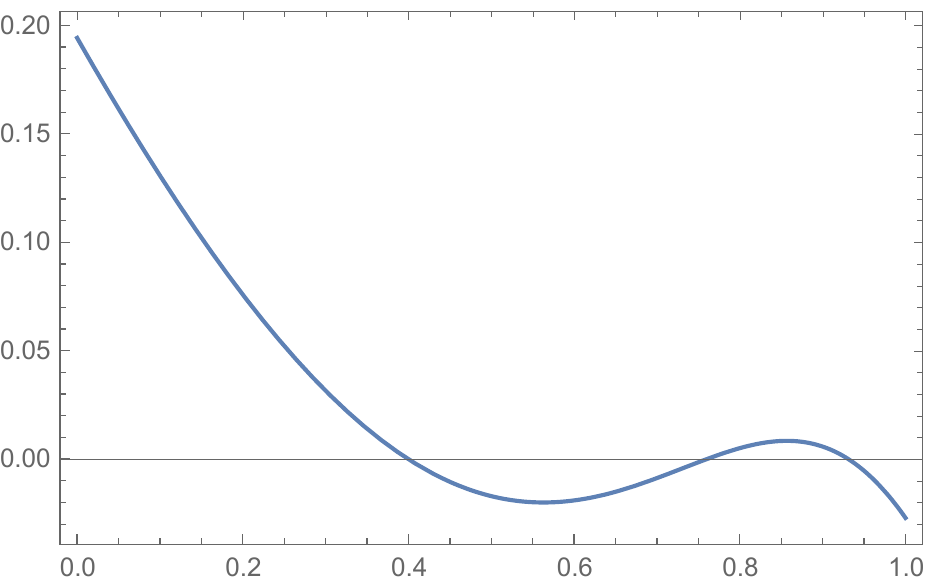}
    \caption{The geometric case 
    of Example \ref{ex:geometric}
    with $q=0.91\in[q_2,q_3)$. 
    As in Figure \ref{fig:04graphs}, we plot the functions
    $h(s)-s$, $h^{(1)}(s)-s$, and $h^{(2)}(s)-s$.
    The functions $h$ and $h^{(1)}$ have unique fixed points,
    but the function $h^{(2)}$ has multiple fixed points; 
    so the tree has probability $1$ to be $1$-stable, 
    but has probability less than $1$ of being $2$-stable. 
    \label{fig:geometric}}
\end{figure}

Except for the transition at $q_0$, the precise
nature and location of all the phase transitions
above are only found numerically. However, with
a sufficiently precise analysis one could rigorously
establish in each case
a smaller interval on which the claimed behaviour holds
(for example we could find some sub-interval of the
claimed interval $(q_2, q_3)$ on which to show
that $h^{(1)}$ has only one fixed point while 
$h^{(2)}$ has more than one fixed point).
\end{example}

In summary, the three families in 
Examples \ref{ex:Poisson}-\ref{ex:geometric}
show a wide variety of behaviours. 
In the Poisson case, one has existence of 
draws whenever one has existence of 
positions with infinite rank. In 
the $0$-or-$4$ case, there is additionally
a phase wth infinite rank vertices but no draws. 
In the geometric case, it is the phase with draws
which is missing; however, one sees additional
phase transitions, losing $3$-stability,
$2$-stability, and $1$-stability step by step
as the parameter increases. 

\bigskip
We end with a question:

\begin{question}
Does there exist for every $k\in\NN$ an offspring distribution for which the \BGW{} tree is $k$-stable with probability $1$, but nonetheless infinite rank positions exist with positive probability?
Numerical explorations have so far only produced examples up to $k=2$ (for example, the Geometric($q$) case
with $q\in(q_1, q_2]$ described above).
\end{question}

\section*{Acknowledgments}
Many thanks to Alexander Holroyd and 
Omer Angel for conversations relating to this work,
particularly during the Montreal summer workshop on
probability and mathematical physics at the CRM in July 2018. 
I am grateful to an anonymous referee, whose comments have considerably 
improved the paper. 
\bibliography{locallyfinite}
\bibliographystyle{abbrv}

\end{document}